\documentclass[12pt,a4paper,oneside]{amsart}

\usepackage{amsmath} 
\usepackage{amssymb} 
\usepackage{amsthm} 
\usepackage[english]{babel} 
\usepackage[font=small,justification=centering]{caption} 
\usepackage[nodayofweek]{datetime}
\usepackage[shortlabels]{enumitem} 
\usepackage[T1]{fontenc} 
\usepackage[a4paper]{geometry} 
\usepackage[utf8]{inputenc} 
\usepackage{ifthen} 
\usepackage{mathabx} 
\usepackage{mathtools} 
\usepackage[dvipsnames]{xcolor} 
\usepackage[pdftex,  colorlinks=true]{hyperref} 
    \hypersetup{urlcolor=RoyalBlue, linkcolor=RoyalBlue,  citecolor=black}
\usepackage{setspace} 
\usepackage{tikz-cd} 
\usepackage{xfrac} 
\usepackage{cleveref}

\makeatletter
\def\@tocline#1#2#3#4#5#6#7{\relax
  \ifnum #1>\c@tocdepth 
  \else
    \par \addpenalty\@secpenalty\addvspace{#2}%
    \begingroup \hyphenpenalty\@M
    \@ifempty{#4}{%
      \@tempdima\csname r@tocindent\number#1\endcsname\relax
    }{%
      \@tempdima#4\relax
    }%
    \parindent\z@ \leftskip#3\relax \advance\leftskip\@tempdima\relax
    \rightskip\@pnumwidth plus4em \parfillskip-\@pnumwidth
    #5\leavevmode\hskip-\@tempdima
      \ifcase #1
       \or\or \hskip 1em \or \hskip 2em \else \hskip 3em \fi%
      #6\nobreak\relax
    \dotfill\hbox to\@pnumwidth{\@tocpagenum{#7}}\par
    \nobreak
    \endgroup
  \fi}
\makeatother

\newtheorem{thm}{Theorem}[section]
\newtheorem{lemma}[thm]{Lemma}

\newtheorem{thmx}{Theorem}
\newtheorem{corx}[thmx]{Corollary}

\theoremstyle{definition}

\newtheorem{remark}[thm]{Remark}

\newcommand{\caln}{{\mathcal{N}}}
\newcommand{\calo}{{\mathcal{O}}}

\newcommand{\calu}{{\mathcal{U}}}

\newcommand{\cd}{\mathrm{cd}}

\newcommand{\hd}{\mathrm{hd}}

\newcommand{\CC}{\mathbb{C}}

\newcommand{\QQ}{\mathbb{Q}}
\newcommand{\FF}{\mathbb{F}}

\usepackage{tikz}
\usetikzlibrary{arrows,quotes}
\tikzstyle{blackNode}=[fill=black, draw=black, shape=circle]

\usepackage[style=alphabetic, citestyle=alphabetic, sorting=nyvt]{biblatex}
\usepackage{csquotes}
\AtEveryBibitem{\ifentrytype{article}{\clearfield{url}}{}}
\AtEveryBibitem{\ifentrytype{article}{\clearfield{ISSN}}{}}

\title[Rational homological dimension in positive characteristic]{A note on the rational homological dimension of lattices in positive characteristic}
\usepackage[foot]{amsaddr}
\author{Sam Hughes}
\address{Mathematical Institute, Andrew Wiles Building, University of Oxford, Oxford OX2 6GG, UK}
\email{sam.hughes@maths.ox.ac.uk}
\date{\today}

\begin{document}

\maketitle

\begin{abstract}
We show via $\ell^2$-homology that the rational homological dimension of a lattice in a product of simple simply connected Chevalley groups over global function fields is equal to the rational cohomological dimension and to the dimension of the associated Bruhat--Tits building.
\end{abstract}

\section{Introduction}
Let $k$ be the function field of an irreducible projective smooth curve $C$ defined over a finite field $\FF_q$. Let $S$ be a finite non-empty set of (closed) points of $C$. Let $\calo_S$ be the ring of rational functions whose poles lie in $S$. For each $p\in S$ there is a discrete valuation $\nu_x$ of $k$ such that $\nu_p(f)$ is the order of vanishing of $f$ at $p$. The valuation ring $\calo_p$ is the ring of functions that do not have a pole at $p$, that is \[\calo_S=\bigcap_{p\not\in S}\calo_p.\]

Let $\bar{k}$ denote the algebraic closure of $k$. Let $\mathbf{G}$ be an affine group scheme defined over $\bar{k}$ such that $\mathbf{G}(\bar{k})$ is almost simple.  For each $p\in S$ there is a completion $k_p$ of $k$ and the group $\mathbf{G}(k_p)$ acts on the Bruhat--Tits building $X_p$.  Thus, we may embed $\mathbf{G}(\calo_S)$ diagonally into the product $\prod_{p\in S}\mathbf{G}(k_p)$ as an arithmetic lattice.

The \emph{rational cohomological dimension} of a group $\Gamma$ is defined to be \[\cd_\QQ(\Gamma)\coloneqq \sup\{n\colon H^n(\Gamma;M)\neq0,\ M\text{ a }\QQ\Gamma\text{-module} \},\]
the \emph{rational homological dimension} is defined completely analogously as
\[\hd_\QQ(\Gamma)\coloneqq \sup\{n\colon H_n(\Gamma;M)\neq0,\ M\text{ a }\QQ\Gamma\text{-module} \}.\]

In \cite{Gandini2012} it is shown that $\cd_\QQ(\mathbf{G}(\calo_S))=\prod_{p\in S}\dim X_p$.  In light of this Ian Leary asked the author what is $\hd_\QQ(\mathbf{G}(\calo_S))$? 

\begin{thmx}\label{thm.hd}
Let $\mathbf{G}$ be a simple simply connected Chevalley group.  Let $k$ and $\calo_S$ be as above, then \[\hd_\QQ(\mathbf{G}(\calo_S))=\cd_\QQ(\mathbf{G}(\calo_S))=\prod_{p\in S}\dim X_p.\]
\end{thmx}

More generally, we obtain the following.

\begin{corx}\label{cor.hd}
Let $\Gamma$ be a lattice in a product of simple simply connected Chevalley groups over global function fields with associated Bruhat--Tits building $X$, then $\hd_\QQ(\Gamma)=\cd_\QQ(\Gamma)=\dim X$.
\end{corx}

The author expects these results are well-known, however, they do not appear in the literature so we take the opportunity to record them here.

\subsection*{Acknowledgements}
The author would like to thank his PhD supervisor Professor Ian Leary for his guidance, support, and suggesting of the question.  This note contains material from the author's PhD thesis \cite{HughesThesis} and was originally part of \cite{Hughes2021a}, but was split off into a number of companion papers \cite{Hughes2021b,Hughes2022} (see also \cite{HughesValiunas2022}) at the request of the referee.  This work was supported by the Engineering and Physical Sciences Research Council grant number 2127970.  The author would like to thank the referee for a number of helpful comments.

\section{\texorpdfstring{$\ell^2$}{l2}--homology and measure equivalence}\label{sec.lat.L2}

Let $\Gamma$ be a group. Both $\Gamma$ and the complex group algebra $\CC \Gamma$ act by left multiplication on the Hilbert space $\ell^2\Gamma$ of square-summable sequences. The \emph{group von Neumann algebra} $\caln \Gamma$ is the ring of $\Gamma$-equivariant bounded operators on $\ell^2G$. The non-zero divisors of $\caln G$ form an Ore set and the Ore localization of $\caln \Gamma$ can be identified with the \emph{ring of affiliated operators} $\calu \Gamma$. 

There are inclusions $\QQ\Gamma\subseteq \caln  \Gamma\subseteq\ell^2\Gamma\subseteq\calu \Gamma$ and it is also known that $\calu \Gamma$ is a self-injective ring which is flat over $\caln \Gamma$. For more details concerning these constructions we refer the reader to \cite{LuckBook} and especially to Theorem 8.22 of Section~8.2.3 therein.  The \emph{von Neumann dimension} and the basic properties we need can be found in \cite[Section~8.3]{LuckBook}.

The $\ell^2$-Betti numbers of a group $\Gamma$, denoted $b_i^{(2)}(\Gamma)$, are then defined to be  the von-Neumann dimensions of the homology groups $H_i(\Gamma;\calu\Gamma)$.  The following lemma is a triviality.

\begin{lemma}\label{lemma.UG.DimensionBound}
Let $\Gamma$ be a discrete group and suppose that $b^{(2)}_i(\Gamma)>0$, then the homology group $H_i(\Gamma;\calu\Gamma)$ is non-trivial.
\end{lemma}

Two countable groups $\Gamma$ and $\Lambda$ are said to be \emph{measure equivalent} if there exist commuting, measure-preserving, free actions of $\Gamma$ and $\Lambda$ on some infinite Lebesgue measure space $(\Omega,m)$, such that the action of each of the groups $\Gamma$ and $\Lambda$ admits a finite measure fundamental domain.  The key examples of measure equivalent groups are lattices in the same locally-compact group \cite{Gromov1993}.  The relevance of this for us is the following deep theorem of Gaboriau.

\begin{thm}[Gaboriau's Theorem \cite{Gaboriau2002}]
Suppose a discrete group $\Gamma$ is measure equivalent to a discrete group $\Lambda$, then $b_p(\Gamma)=0$ if and only if $b_p(\Lambda)=0$.
\end{thm}

\section{Proofs}

\begin{proof}[Proof of Theorem~\ref{thm.hd}]
We first note that the group $\Gamma:=\mathbf{G}(\calo_S)$ is measure equivalent to the product $\Lambda:=\prod_{p\in S}\mathbf{G}(\FF_q[t_p])$ for some suitably chosen $t_p\in \calo_p$.  By \cite[Theorem~1.6]{PetersonSauerThom2018} (see also \cite{Dymara2004,Dymara2006,DavisDymaraJanuskiewicz2007}) the group $\mathbf{G}(\FF_q[t_p])$ has one non-vanishing $\ell^2$-Betti number in dimension $\dim(X_p)$.  Hence, by the K\"unneth formula $\Lambda$ has one non-vanishing $\ell^2$-Betti number in dimension $d=\prod_{p\in S}\dim X_p$  Thus, by Gaboriau's theorem, the group $\Gamma$ has exactly one non-vanishing $\ell^2$-Betti number in dimension $d$.  It follows from Lemma~\ref{lemma.UG.DimensionBound} that $\hd_\QQ(\Gamma)\geq d$.  The reverse inequality follows from the fact that $\Gamma$ acts properly on the $d$-dimensional space $\prod_{p\in S}\dim X_p$.
\end{proof}

\begin{proof}[Proof of Corollary~\ref{cor.hd}]
The proof of the corollary is entirely analogous.  First, we split $\mathbf{G}$ into a product of simple groups $\prod_{i=1}^n\mathbf{G}_i$ corresponding to the decomposition of the Bruhat--Tits building $X=\prod_{i=1}^nX_i$.  Let $\Lambda_i$ be a lattice in $\mathbf{G}_i$ and let $\Lambda=\prod_{i=1}^n\Lambda_i$.  Each $\Lambda_i$ has a non-vanishing $\ell^2$-Betti Number in dimension $\dim X_i$.  In particular, $\Lambda$ has a non-vanishing $\ell^2$-Betti Number in dimension $\dim X=\prod_{i=1}^n\dim X_i$.  By Gaboriau's Theorem $\Gamma$ also has non-vanishing $\ell^2$-Betti Number in dimension $\dim X$.  It follows from Lemma~\ref{lemma.UG.DimensionBound} that $\hd_\QQ(\Gamma)\geq d$.  The reverse inequality follows from the fact that $\Gamma$ acts properly on the $d$-dimensional space $\prod_{p\in S}\dim X_p$.
\end{proof}

\begin{remark}
A similar argument can be applied to lattices in products of simple simply-connected algebraic groups over locally compact $p$-adic fields.  One obtains the analogous result for such a lattice $\Gamma$ that $\cd_\QQ (\Gamma)=\hd_\QQ (\Gamma)=\dim X$, where $X$ is the associated Bruhat--Tits building.
\end{remark}

\AtNextBibliography{\small}
\printbibliography

@article {DavisDymaraJanuskiewicz2007,
    AUTHOR = {Davis, Michael W. and Dymara, Jan and Januszkiewicz, Tadeusz
              and Okun, Boris},
     TITLE = {Weighted {$L^2$}-cohomology of {C}oxeter groups},
   JOURNAL = {Geom. Topol.},
  FJOURNAL = {Geometry \& Topology},
    VOLUME = {11},
      YEAR = {2007},
     PAGES = {47--138},
      ISSN = {1465-3060},
   MRCLASS = {20F55 (20C08 20F65 20J06 46L10 57M07)},
  MRNUMBER = {2287919},
MRREVIEWER = {Alain Valette},
       DOI = {10.2140/gt.2007.11.47},
      URL = {https://doi.org/10.2140/gt.2007.11.47},
}

@article {Dymara2006,
    AUTHOR = {Dymara, Jan},
     TITLE = {Thin buildings},
   JOURNAL = {Geom. Topol.},
  FJOURNAL = {Geometry and Topology},
    VOLUME = {10},
      YEAR = {2006},
     PAGES = {667--694},
      ISSN = {1465-3060},
   MRCLASS = {20E42 (20F55 51F15 57P10)},
  MRNUMBER = {2240901},
MRREVIEWER = {Alain Valette},
       DOI = {10.2140/gt.2006.10.667},
      URL = {https://doi.org/10.2140/gt.2006.10.667},
}

@article {Dymara2004,
    AUTHOR = {Dymara, Jan},
     TITLE = {{$L^2$}-cohomology of buildings with fundamental class},
   JOURNAL = {Proc. Amer. Math. Soc.},
  FJOURNAL = {Proceedings of the American Mathematical Society},
    VOLUME = {132},
      YEAR = {2004},
    NUMBER = {6},
     PAGES = {1839--1843},
      ISSN = {0002-9939},
   MRCLASS = {58J22 (20F55 20F65)},
  MRNUMBER = {2051148},
MRREVIEWER = {Alain Valette},
       DOI = {10.1090/S0002-9939-03-07234-4},
      URL = {https://doi.org/10.1090/S0002-9939-03-07234-4},
}

@article {Gaboriau2002,
    AUTHOR = {Gaboriau, Damien},
     TITLE = {Invariants {$l^2$} de relations d'\'{e}quivalence et de groupes},
   JOURNAL = {Publ. Math. Inst. Hautes \'{E}tudes Sci.},
  FJOURNAL = {Publications Math\'{e}matiques. Institut de Hautes \'{E}tudes
              Scientifiques},
    VOLUME = {95},
      YEAR = {2002},
     PAGES = {93--150},
      ISSN = {0073-8301},
   MRCLASS = {22D40 (22F10 28D15 37A15 37A20)},
  MRNUMBER = {1953191},
MRREVIEWER = {Bachir Bekka},
       DOI = {10.1007/s102400200002},
      URL = {https://doi.org/10.1007/s102400200002},}

@incollection {Gromov1993,
    AUTHOR = {Gromov, M.},
     TITLE = {Asymptotic invariants of infinite groups},
 BOOKTITLE = {Geometric group theory, {V}ol. 2 ({S}ussex, 1991)},
    SERIES = {London Math. Soc. Lecture Note Ser.},
    VOLUME = {182},
     PAGES = {1--295},
 PUBLISHER = {Cambridge Univ. Press, Cambridge},
      YEAR = {1993},
   MRCLASS = {20F32 (57M07)},
  MRNUMBER = {1253544},
}

@book {LuckBook,
    AUTHOR = {L{\"{u}}ck, Wolfgang},
     TITLE = {{$L^2$}-invariants: theory and applications to geometry and
              {$K$}-theory},
    SERIES = {Ergebnisse der Mathematik und ihrer Grenzgebiete. 3. Folge. A
              Series of Modern Surveys in Mathematics [Results in
              Mathematics and Related Areas. 3rd Series. A Series of Modern
              Surveys in Mathematics]},
    VOLUME = {44},
 PUBLISHER = {Springer-Verlag, Berlin},
      YEAR = {2002},
     PAGES = {xvi+595},
      ISBN = {3-540-43566-2},
   MRCLASS = {58J22 (19K56 46L80 57Q10 57R20 58J52)},
  MRNUMBER = {1926649},
MRREVIEWER = {Thomas Schick},
       DOI = {10.1007/978-3-662-04687-6},}

@article {PetersonSauerThom2018,
    AUTHOR = {Petersen, Henrik Densing and Sauer, Roman and Thom, Andreas},
     TITLE = {{$L^2$}-{B}etti numbers of totally disconnected groups and
              their approximation by {B}etti numbers of lattices},
   JOURNAL = {J. Topol.},
  FJOURNAL = {Journal of Topology},
    VOLUME = {11},
      YEAR = {2018},
    NUMBER = {1},
     PAGES = {257--282},
      ISSN = {1753-8416},
   MRCLASS = {20D05 (20E15 22D15 22D25 22D99 55N35)},
  MRNUMBER = {3784232},
MRREVIEWER = {Takahiro Sudo},
       DOI = {10.1112/topo.12056},}

@article {Gandini2012,
    AUTHOR = {Gandini, Giovanni},
     TITLE = {Bounding the homological finiteness length},
   JOURNAL = {Bull. Lond. Math. Soc.},
  FJOURNAL = {Bulletin of the London Mathematical Society},
    VOLUME = {44},
      YEAR = {2012},
    NUMBER = {6},
     PAGES = {1209--1214},
      ISSN = {0024-6093},
   MRCLASS = {20F65 (20J05)},
  MRNUMBER = {3007653},
MRREVIEWER = {Nansen Petrosyan},
       DOI = {10.1112/blms/bds047},}

@phdthesis{HughesThesis,
title={Equivariant cohomology, lattices, and trees},
author={Sam Hughes},
year={2021},
school={School of Mathematical Sciences, University of Southampton}
}

@article{Hughes2021b,
      title={Lattices in a product of trees, hierarchically hyperbolic groups, and virtual torsion-freeness}, 
      author={Sam Hughes},
      year={2021},
      eprint={2105.02847},
      archivePrefix={arXiv},
      primaryClass={math.GR},
      pubstate = {To appear},
      journal={Bull. London Math. Soc.},
      fjournal={Bulletin of the London Mathematical Society}
}

@misc{Hughes2022,
      title={Irreducible lattices fibring over the circle}, 
      author={Sam Hughes},
      year={2022},
      eprint={2201.06525},
      archivePrefix={arXiv},
      primaryClass={math.GR}
}

@misc{Hughes2021a,
      title={Graphs and complexes of lattices}, 
      author={Sam Hughes},
      year={2021},
      eprint={2104.13728},
      archivePrefix={arXiv},
      primaryClass={math.GR}
}

@misc{HughesValiunas2022,
      title={Commensurating HNN-extensions: hierarchical hyperbolicity and biautomaticity}, 
      author={Sam Hughes and Motiejus Valiunas},
      year={2021},
      eprint={2203.11996},
      archivePrefix={arXiv},
      primaryClass={math.GR}
}

\end{document}